\documentclass[12pt]{article}
\usepackage{jhep-mod}
\usepackage{bm}
\usepackage{amssymb,amsmath,amsthm}
\usepackage{mathrsfs}
\usepackage[utf8]{inputenc}
\usepackage{enumerate}
\hypersetup{colorlinks=true} 
\usepackage{appendix}
\usepackage{graphicx}
\usepackage{dcolumn}
\usepackage{bm}
\usepackage{multirow}
\usepackage{float}
\usepackage{tikz}
\usepackage[normalem]{ulem}
\definecolor{purple}{rgb}{1,0,1}
\definecolor{lime}{HTML}{A6CE39} % needs xcolor

\newtheorem{theorem}{Theorem}
\newtheorem{corollary}{Corollary}
\definecolor{lime}{HTML}{A6CE39}
\newcommand{\orcidicon}{%
	\begin{tikzpicture}
	\draw[lime, fill=lime] (0,0) 
		circle [radius=0.16] 
		node[white] {{\fontfamily{qag}\selectfont \tiny ID}};
	\draw[white, fill=white] (-0.0625,0.095) 
		circle [radius=0.007];
	\end{tikzpicture}
	\hspace{-2mm}
}
\newcommand\orcidMatt{{\href{https://orcid.org/0000-0003-1088-6485}{\orcidicon}}}
\begin{document}
%========================================================

\title{\huge {Variants on Andrica's conjecture 
with and without the Riemann hypothesis}}

\author{\Large Matt Visser\orcidMatt{}}
%========================================================%========================================================
%========================================================%========================================================
\affiliation{School of Mathematics and Statistics, Victoria University of Wellington, \\
PO Box 600, Wellington 6140, New Zealand}
%========================================================%========================================================
\emailAdd{matt.visser@sms.vuw.ac.nz}
%========================================================%========================================================

\abstract{

\parindent0pt
\parskip7pt
The gap between what we can explicitly prove regarding the distribution of primes and what we suspect regarding the distribution of primes is enormous.
It is (reasonably) well-known that the Riemann hypothesis is not sufficient to  prove Andrica's conjecture: $\forall n\geq 1$,  is $\sqrt{p_{n+1}}-\sqrt{p_n} \leq 1$? But can one at least get tolerably close? I shall first show that with a logarithmic modification, provided one assumes the Riemann hypothesis, one has
\[
{\sqrt{p_{n+1}}\over\ln p_{n+1}} -{\sqrt{p_n}\over\ln p_n} < {11\over25};
\qquad (n\geq1).
\]
Then, by considering more general $m^{th}$ roots, again assuming the Riemann hypothesis, I shall show that
\[
{\sqrt[m]{p_{n+1}}} -{\sqrt[m]{p_n}} < {44\over25 \,e\, (m-2)}; \qquad (n\geq 3;\;  m >2). 
\]
In counterpoint, if we limit ourselves to what we can currently prove unconditionally, then the only explicit Andrica-like results seem to be variants on the results below:
\[
\ln^2 p_{n+1} -  \ln^2 p_n < 9;
\qquad (n\geq1).
\]
\[
\ln^3 p_{n+1} -  \ln^3 p_n < 52;
\qquad (n\geq1).
\]
\[
\ln^4 p_{n+1} -  \ln^4 p_n < 991;
\qquad (n\geq1).
\]
I shall also slightly update the region on which Andrica's conjecture is unconditionally verified.

\small
\bigskip
{\sc Date:} 25 November 2018; \LaTeX-ed \today
\\
{\sc Sequences:} A005250 A002386 A005669 A000101 A107578  
\\
{\sc Published:} Mathematics {\bf6 \#12} (2018) 289.
}

\notoc
\maketitle

%========================================================
\def\tr{{\mathrm{tr}}}
\def\diag{{\mathrm{diag}}}
\parindent0pt
\parskip7pt
%========================================================
\section{Introduction}
%========================================================
%---------------------------------------------------------------------------------------------------------------------------------------------
\label{S:intro}
%---------------------------------------------------------------------------------------------------------------------------------------------

The Riemann hypothesis continues to defeat all attempts to prove or disprove it~\cite{Borwein, Ivic, Patterson, Rogers, Iwaniec, Tenenbaum}. On the other hand assuming the Riemann hypothesis provides a wealth of intriguing tentative suggestions regarding the distribution of the prime numbers~\cite{Borwein, Ivic, Patterson, Rogers, Iwaniec, Tenenbaum}. However, even the Riemann hypothesis is insufficient to prove Andrica's conjecture~\cite{Andrica,Andrica-online,Wells}:  
\begin{equation}
\forall n\geq 1,  \quad\hbox{is}\quad \sqrt{p_{n+1}}-\sqrt{p_n} \leq 1? 
\end{equation}
(In section \ref{S:unconditional} below, by considering known maximal prime gaps, I shall numerically [and unconditionally] verify Andrica's conjecture up to just below $81^{st}$ maximal prime gap; certainly for all primes less than $1.836\times 10^{19}$.)
Somewhat weaker results that have been unconditionally proved include Sandor's 1985 result~\cite{Sandor}

\begin{equation}
\liminf_{n\to\infty} \sqrt[4]{p_n} \left(  \sqrt{p_{n+1}}-\sqrt{p_n}\right) =0,
\end{equation}
and the more recent 2017 result by Lowry--Duda~\cite{Lowry-Duda} that for $\alpha>0$, $\beta>0$, and $ \alpha+\beta<1$:

\begin{equation}
\lim \inf \sqrt[\beta]{p_n} \left(  \sqrt[\alpha]{p_{n+1}}-\sqrt[\alpha]{p_n}\right) =0.
% \qquad (\alpha>0; \quad \beta>0; \quad \alpha+\beta<1).
\end{equation}
Other recent articles on the distribution of primes include~\cite{Kristyan:2017, Kristyan:2018}.

In the current article I shall seek to derive results as close to Andrica's conjecture as possible. 
The basic tools we use are based on the behaviour of prime gaps under the Riemann hypothesis.
I shall primarily make use of the recent fully explicit bound due to Carneiro, Milinovich, and Soundararajan~\cite{Carneiro:2017}:
\begin{theorem}[Prime gaps; Carneiro--Milinovich--Soundararajan]\label{T:prime}\ \\
Assuming the Riemann hypothesis,
\begin{equation}
\forall n\geq 3, \; p_n \geq 5, \qquad  g_n := p_{n+1}-p_n < {22\over25}\, \sqrt{p_n}\,\ln p_n.
\end{equation}
\end{theorem}
(In section \ref{S:CMS} below, by considering known maximal prime gaps, I shall numerically [and unconditionally] verify this result up to just below the $81^{st}$ maximal prime gap; certainly for all primes less than $1.836\times 10^{19}$.)
The ``close to Andrica'' results I will prove below are these:
\begin{theorem}[Logarithmic modification of Andrica]\label{T:log}
Assuming the Riemann hypothesis,
\begin{equation}
\forall n\geq 1, \qquad \;{\sqrt{p_{n+1}}\over\ln p_{n+1}} -{\sqrt{p_n}\over\ln p_n} \leq {11\over25}.
\end{equation}
\end{theorem}
\begin{theorem}[Higher root modification of Andrica]\label{T:root}
Assuming the Riemann hypothesis,
\begin{equation}
\forall n\geq 3, \; \forall m >2, \qquad  {\sqrt[m]{p_{n+1}}} -{\sqrt[m]{p_n}} \leq {44\over25 \,e\,(m-2)}.
\end{equation}
\end{theorem}

%------------------------------------------------
\section{Known results on prime gaps assuming  Riemann hypothesis}\label{S:known}
%------------------------------------------------
Two older theorems addressing the issue of prime gaps are ``ineffective''  (meaning one or more implicit constants are known to be finite but are otherwise undetermined):
\begin{theorem}[Cramer 1919~\cite{Cramer:1919,Cramer:1936}]
Assuming the Riemann hypothesis,
\begin{equation}
g_n := p_{n+1}-p_n = O(\sqrt{p_n}\ln(p_n)).
\end{equation}
\end{theorem}
Unfortunately, this particular theorem only gives qualitative, not quantitative, information.

\begin{theorem}[Goldston 1982~\cite{Goldston}]
Assuming the Riemann hypothesis,
\begin{equation}
g_n := p_{n+1}-p_n\leq 4 \sqrt{p_n}\ln(p_n);      \qquad n \hbox{ sufficiently large}.
\end{equation}
\end{theorem}
This particular theorem gives quantitative information about the size of prime gaps, but only qualitative information as to the domain of validity.
The considerably more recent Carneiro, Milinovich, and Soundararajan theorem~\cite{Carneiro:2017},  presented as theorem 1 above, is fully explicit. It is that theorem that will be my primary tool.

%------------------------------------------------
\section{Logarithmic modification of Andrica: Proof of Theorem \ref{T:log} }\label{S:log}
%------------------------------------------------

The proof of theorem \ref{T:log}, given our input assumptions, is straightforward.\\
Consider the function $\sqrt{p}/\ln p$ and note
\begin{equation}
\left(\sqrt{p}\over\ln p\right)' =   {\ln  p - 2\over 2 \sqrt{p} \ln^2 p} > 0 \qquad (p > e^2\approx 7.39);
\end{equation}
\begin{equation}
\left(\sqrt{p}\over\ln p\right)'' = - {\ln^2  p - 8\over 4 p^{3/2} \ln^3 p} < 0 \qquad (p > e^{\sqrt{8}}\approx 16.92).
\end{equation}
So $\sqrt{p}/\ln p$ is certainly monotone and convex for $p\geq 17$.
Thence for $ p_n \geq 17$, that is $n\geq 7$, we have
\begin{equation}
{\sqrt{p_{n+1}}\over\ln p_{n+1}}- {\sqrt{p_n}\over\ln p_n} <  {\ln  p_n - 2\over2 \sqrt{p_n} \ln^2 p_n} g_n
<  {11\over25} \left(\ln  p_n - 2\over \ln p_n\right) ={11\over25}\left( 1-{2\over p_n}\right) < {11\over25}.
\end{equation}
(The first step in this chain of inequalities is based on convexity, the second step on the Carneiro--Milinovich--Soundararajan result; the remaining steps are trivial.)
A quick verification shows that this also holds for $n\in\{1,2,3,4,5,6\}$ and so we have
\begin{equation}
{\sqrt{p_{n+1}}\over\ln p_{n+1}} - {\sqrt{p_n}\over\ln p_n} < {11\over25};   \qquad\qquad (n\geq 1).
\end{equation}
This by no means an optimal bound, but it is, given the Riemann hypothesis and the Carneiro--Milinovich--Soundararajan prime gap result it implies, both easy to establish and easy to work with. 

%------------------------------------------------
\section{Higher-root modification of Andrica: Proof of Theorem \ref{T:root}}\label{S:root}
%------------------------------------------------

The proof of theorem \ref{T:root}, given our input assumptions, is straightforward.\\
Consider the function $\sqrt[m]{p} = p^{1/m}$ and restrict attention to $p>0$ and $m>1$. Then
\begin{equation}
( p^{1/m})' = {1\over m} \, p^{1/m-1} >0;  \qquad ( p^{1/m})''   =  -{m-1\over m^2} \, p^{1/m-2} <0;  
\end{equation}
Thus the function $\sqrt[m]{p} = p^{1/m}$  is monotone and convex for $p>0$ and $m>1$. We have
\begin{equation}
\sqrt[m]{p_{n+1}}- \sqrt[m]{p_n} <    {1\over m}  \,p_n^{1/m-1}  \; g_n < {22\over 25 \,m}  \, p_n^{1/m-1/2} \; \ln(p_n); 
\qquad    \qquad (n\geq 3, m>1).
\end{equation}
(The first step in this chain of inequalities is based on convexity, the second step on the Carneiro--Milinovich--Soundararajan result.)
If $m\leq2$ this result is true but not particularly interesting. For $m>2$ the function $p^{1/m-1/2} \,\ln p$ rises from zero (at $p=1$) to a maximum, and subsequently dies back to zero asymptotically as $p\to \infty $. The maximum occurs at $p_\mathrm{critical}= \exp(2m/(m-2))$ where the function takes on the value $(p^{1/m-1/2} \ln p)_\mathrm{max}= 2me^{-1}/(m-2)$, thereby implying
\begin{equation}
\sqrt[m]{p_{n+1}}- \sqrt[m]{p_n} < {44\over25} \; { e^{-1}\over m-2}\;;  \qquad \qquad (n\geq 3; \; m>2).
\end{equation}
This is the result we were seeking to establish.
This is again by no means an optimal bound, but it is, given the Riemann hypothesis and the Carneiro--Milinovich--Soundararajan prime gap result it implies, both easy to establish and easy to work with. 

Note that for $m=2$, the situation relevant to the standard Andrica conjecture, we merely have
\begin{equation}
\sqrt{p_{n+1}}- \sqrt{p_n} <   {11\over25}\;\ln(p_n);    \qquad (n\geq 3).
\end{equation}
For $m=2$ this is not enough to conclude anything useful regarding the standard Andrica conjecture.

\noindent
In contrast for $m=3$ and $m=4$ we certainly have:
\begin{corollary}[Cube root modification of Andrica]\label{C:cube-root}
Assuming the Riemann hypothesis,
\begin{equation}
\forall n\geq 1, \qquad {\sqrt[3]{p_{n+1}}} -{\sqrt[3]{p_n}} \leq {44\over25 e} < {13\over20}.
\end{equation}
\end{corollary}
\begin{corollary}[Fourth root modification of Andrica]\label{C:fourth-root}
Assuming the Riemann hypothesis,
\begin{equation}
\forall n\geq 1, \qquad {\sqrt[4]{p_{n+1}}} -{\sqrt[4]{p_n}} \leq {22\over25 e} < {13\over40}.
\end{equation}
\end{corollary}

%-----------------------------------------------
\section{Unconditional results for Andrica--like variants}
%------------------------------------------------
I shall now focus on some much weaker but unconditional results.
A number of explicit unconditional theorems on the occurrence of primes in small gaps are of the following general form.
\begin{theorem}[Primes in short intervals]\label{T:intervals}
For $x>x_0$, $m\geq0$, and some explicit $C>0$, there is always at least one prime in the interval:
\begin{equation}
\left(x, x+ C\;{x\over \ln^m c} \right).
\end{equation}
Specifically, we have:
%%%%%%%%%%%%%%%%%%%%%%%%%%%%%%%%%%%%%%%%%%
%\vspace{-10pt}
\emph{
\begin{center}
\begin{tabular}{||c|c|c||c|c||}
\hline\hline
$x_0$ & $m$ & $C$ & Reference & arXiv \\
\hline\hline
396738 & 2 & ${1\over25}$ & Dusart 2010~\cite{Dusart:2010} & 1002.0442\\
2898329 & 2 & ${1\over 111}$ & Trudgian 2014~\cite{Trudigan:2014} & 1401.2689\\
468991632 & 2 & ${1\over5000}$ & Dusart 2018~\cite{Dusart2} &--- \\
89693  &3 & 1  & Dusart 2018~\cite{Dusart2} & --- \\
 6034256& 3& 0.087 &Axler 2017~\cite{Axler:2017} &1703.08032\\
1& 4 &198.2 & Axler 2017~\cite{Axler:2017} & 1703.08032\\
\hline\hline
\end{tabular}
\end{center}
}
%%%%%%%%%%%%%%%%%%%%%%%%%%%%%%%%%%%%%%%%%%%

\end{theorem}
Taking $x = p_n$ in Theorem \ref{T:intervals}, this becomes a bound on the prime gaps.
\begin{theorem}[Prime gaps]\label{T:gaps}
For $p_n>x_0$, that is $n > \pi(x_0)$, and taking $m\geq0$ and $C>0$ from Theorem~\ref{T:intervals}, the prime gaps are bounded by:
\begin{equation}
g_n := p_{n+1}-p_n <  C \; {p_n\over \ln^m p_n}.
\end{equation}
\end{theorem}
Now consider the function $(\ln^{m+1} p)$ and note
\begin{equation}
(\ln^{m+1} p)' = {(m+1)\ln^{m} p\over p} >0; 
 \qquad (m+1>0, \;p >1);
\end{equation}
\begin{equation}
(\ln^{m+1} p)'' = -{(m+1)\ln^{m-1} p\;  (\ln p-m)\over p^2}<0; 
\qquad (m+1>0, \; p >e^{m}).
\end{equation}
Then, in view of the convexity of $\ln^m p$ on the specified domain, we have
\begin{eqnarray}
(\ln^{m+1} p_{n+1}) - (\ln^{m+1} p_n) 
&<&  {(m+1)\ln^{m} p_n\over p_n} \; g_n \\
&<&{(m+1)\ln^{m} p_n\over p_n}  \; C \; {p_n\over \ln^m p_n} \\[2pt]
&=&(m+1) \; C.
\end{eqnarray}

That is, we have demonstrated:
\begin{theorem}[Explicit unconditional bounds]\label{T:bounds}\ \\
For $p_n>\max\{x_0,e^m\}$, that is, 
$n > \pi(\max\{x_0,e^m\})$, and taking $m\geq0$ and $C>0$ from the table in Theorem~\ref{T:intervals}, we have the explicit unconditional bounds:
\begin{equation}
(\ln^{m+1} p_{n+1}) - (\ln^{m+1} p_n) < (m+1) C.
\end{equation}
\end{theorem}
While weak, these bounds are both unconditional and fully explicit. 

\begin{corollary}
\begin{equation}
(\ln^{5} p_{n+1}) - (\ln^{5} p_n) < 991;  \qquad  (n\geq1).
\end{equation}
\end{corollary}
\begin{proof}
The general result above implies this for $p_n>e^4$, that is $n>16$, the remaining cases can be checked by explicit calculation.
\end{proof}
\begin{corollary}
\begin{equation}
(\ln^{4} p_{n+1}) - (\ln^{4} p_n) < 0.348 ;  \qquad  (p_n>6034256; \; n>415069).
\end{equation}
\end{corollary}
This bound, for the specified value 0.348, is sharp --- it fails by 5\% for $p_n=6034256$, $n=415069$. I now present a more relaxed bound with a wider range of validity.
\begin{corollary}
\begin{equation}
(\ln^{4} p_{n+1}) - (\ln^{4} p_n) < 52 ;  \qquad  (p_n\geq2; \; n\geq1).
\end{equation}
\end{corollary}
\begin{proof}
Use the previous corollary for $n>415069$ and check $n\leq 415069$  by explicit calculation.
\end{proof}
\begin{corollary}
\begin{equation}
(\ln^{3} p_{n+1}) - (\ln^{3} p_n) < {2\over111} ;  \qquad  (p_n>2898329; \; n>209989).
\end{equation}
\end{corollary}
This bound, for the specified value $2\over111$, is sharp --- it fails by over 50\% for $p_n=2898329$, $n=209989$. 
I now present a more relaxed bound with a wider range of validity.
\begin{corollary}
\begin{equation}
(\ln^{3} p_{n+1}) - (\ln^{3} p_n) < {2\over25} ;  \qquad  (p_n>396738; \; n>33608).
\end{equation}
\end{corollary}
This bound, for the specified value $2\over25$, is sharp --- it fails by over 57\% for $p_n=396738$, $n=33608$.  I now present an even  more relaxed bound with a wider range of validity.
\begin{corollary}
\begin{equation}
(\ln^{3} p_{n+1}) - (\ln^{3} p_n) < 9;  \qquad  (p_n\geq2; \; n\geq1).
\end{equation}
\end{corollary}
\begin{proof}
Use the previous corollary for $n>33608$ and check $n\leq 33608$  by explicit calculation.
\end{proof}
While all relatively weak, all these bounds have the pronounced virtue of being both fully explicit, and completely unconditional. Furthermore is is now clear how to systematically turn prime gap results of the type considered in Theorem~\ref{T:intervals} into Andrica-like bounds of the type considered in Theorem~\ref{T:bounds}.

%-----------------------------------------------
\section{Unconditional numerical results for the standard Andrica conjecture}
%------------------------------------------------
\label{S:unconditional}
%------------------------------------------------
Andrica's conjecture can be rearranged to be equivalent to
\begin{equation}
g_n \leq 2 \sqrt{p_n} +1.
\end{equation}
In an unpublished note (see comments in references~\cite{Wells,Andrica-online}) Imran Ghory rephrased this in terms of maximal prime gaps.
Let the triplet $(i, g^*_i, p^*_i)$ denote the $i^{th}$ maximal prime gap; of width $g^*_i$, starting at the prime $p^*_i$. (See see the sequences A005250, A002386, A005669, A000101, A107578.)
80 such maximal prime gaps are currently known~\cite{g80}, up to $g^*_{80}=1550$ and $p^*_{80}= 18,361,375,334,787,046,697 >1.836\times 10^{19}$.
Imran Ghory observed the equivalent of
\begin{equation}
\forall p_n\in [p^*_i,p^*_j]  \qquad   g_n \leq g^*_j; \qquad  2 \sqrt{p_n} +1 \geq  2 \sqrt{p^*_i} +1.
\end{equation}
That is, Andrica's conjecture certainly holds on the interval $p_n\in [p^*_i,p^*_j]$ if one has
\begin{equation}
g^*_j \leq  2 \sqrt{p^*_i} +1; \qquad \hbox{that is} \qquad \left({g^*_j\over2} -1\right)^2  < p^*_i.
\end{equation}
But this is easily checked to hold on the intervals $[p^*_{19},p^*_{80}]$,   $[p^*_{11},p^*_{19}]$, $[p^*_{7},p^*_{11}]$,
$[p^*_{5},p^*_{7}]$, $[p^*_{4},p^*_{5}]$, $[p^*_{3},p^*_{4}]$, $[p^*_{2},p^*_{3}]$, and $[p^*_{1},p^*_{2}]$. 
Thus Andrica's conjecture certainly holds up to the 80$^{th}$ maximal prime gap,  
$p^*_{80} = 18,361,375,334,787,046,697 >1.836\times 10^{19}$.  This bound will certainly be improved as additional maximal prime gaps are identified. 

As a slight variant on this argument, consider the interval $[p^*_i, p^*_{i+1}-1]$, from the lower end of the $i^{th}$ maximal prime gap to just below the beginning of the $(i+1)^{th}$ maximal prime gap. Then everywhere in this interval
\begin{equation}
\forall p_n\in [p^*_i,p^*_{i+1}-1]  \qquad   g_n \leq g^*_i; \qquad  2 \sqrt{p_n} +1 \geq  2 \sqrt{p^*_i} +1.
\end{equation}
That is, Andrica's conjecture certainly holds on the interval $p_n\in [p^*_i,p^*_{1+1}-1]$ if it holds at the beginning of this interval.
Consequently, explicitly checking the inequality for $p^*_{80}$,  Andrica's conjecture holds unconditionally up to just before the beginning of the 81$^{st}$ maximal prime gap, $p^*_{81}-1$, even if we do not yet know the value of $p^*_{81}$. 

%-----------------------------------------------
\section{Unconditional numerical results for the Carneiro--Milinovich--Soundararajan inequality}
%------------------------------------------------
\label{S:CMS}
%------------------------------------------------
Consider the  Carneiro--Milinovich--Soundararajan inequality
\begin{equation}
g_n := p_{n+1}-p_n < {22\over25}\, \sqrt{p_n}\,\ln p_n,
\end{equation}
and note that the right-hand side is monotone increasing. 
Again consider the interval $[p^*_i, p^*_{i+1}-1]$, from the lower end of the $i^{th}$ maximal prime gap to just below the beginning of the $(i+1)^{th}$ maximal prime gap. 
Then the  Carneiro--Milinovich--Soundararajan inequality certainly holds on the entire interval $p_n\in [p^*_i,p^*_{1+1}-1]$ if it holds at the beginning of this interval.
Explicitly checking the inequality for all the known maximal prime gaps up to $p^*_{80}$,  see reference~\cite{g80}, the  Carneiro--Milinovich--Soundararajan inequality holds unconditionally up to just before the beginning of the 81$^{st}$ maximal prime gap, $p^*_{81}-1$, even if we do not yet know the value of $p^*_{81}$. 
Certainly the the  Carneiro--Milinovich--Soundararajan inequality holds for all primes less than $1.836\times 10^{19}$. 
(Of course, assuming the Riemann hypothesis, Carneiro--Milinovich--Soundararajan have proved a theorem; so this numerical exercise is only a partial numerical check on the input to our arguments.)
This also implies that theorems  2 and 3 are unconditionally numerically verified for all primes less than $1.836\times 10^{19}$. 

%------------------------------------------------
\section{Discussion}\label{S:Discussion}
%------------------------------------------------
While the Riemann hypothesis provides (among very many other things) a nice explicit bound on prime gaps, it is still not quite sufficient to prove Andrica's conjecture --- though as seen above, one can get reasonably close. There are a number of places where the argument might be tightened --- the presentation above was designed to be simple and direct, not necessarily optimal. 
Of course the really big improvement in theorems 1--2--3 would be if any of these results could be made unconditional. 
While the numerical evidence certainly suggests this, a proof seems impossible with current techniques.

In contrast what we can currently prove unconditionally is rather weak; so improving the constants in theorems 6--7--8 would also be of some considerable interest.

%------------------------------------------------
\acknowledgments{
This research was supported by the Marsden Fund, administered by the Royal Society of \\New~Zealand.  
}
%-------------------------------------------------
%========================================================

%========================================================

\begin{thebibliography}{69}  
%========================================================
% GENERAL BACKGROUND
%========================================================
\bibitem{Borwein}
Peter Borwein, Stephen Choi, Brenden Rooney, and Andrea Weirathmueller,\\
\emph{The Riemann hypothesis},
(Canadian Mathematical Society, Spinger, New York, 2008).

\bibitem{Ivic}
Aleksander Ivi\'c, 
\emph{The Riemann zeta-function},
(Dover, New York, 2003).

\bibitem{Patterson}
S.~J.~Patterson,
\emph{An introduction to the theory of the Riemann zeta-function},\\
(Cambridge University Press, Cambridge, England, 1988).

\bibitem{Rogers}
Brad Rodgers and Terence Tao,
``The De Bruijn--Newman constant is non-negative'',
arXiv:1801.05914 [math.NT].

\bibitem{Iwaniec}
Henryk  Iwaniec, 
\emph{Lectures on the Riemann zeta function}, University Lecture Series 62, (American Mathematical Society, Providence, RI, 2014).

\bibitem{Tenenbaum} 
G\'erald  Tenenbaum and Michel Mend\`es France, \\
\emph{The prime numbers and their distribution}, Student Mathematical Library 6, \\
(American Mathematical Society, Providence, RI, 2000).

%========================================================
% ANDRICA
%========================================================

\bibitem{Andrica}
Dorin Andrica, 
``Note on a conjecture in prime number theory'',\\
Studia Univ. Babes--Bolyai Math. {\bf 31 \# 4} (1986) 44--48. 
ISSN 0252-1938. 
%Zbl 0623.10030.

\bibitem{Andrica-online}
Andrica's conjecture appears to have been verified numerically up to $p_n \leq 4 \times 10^{18}$:

\url{https://en.wikipedia.org/wiki/Andrica\%27s\_conjecture}

The key step seems to have been an unpublished observation by Imran Ghory.

See also reference~\cite{Wells}, and the discussion in section~\ref{S:unconditional} of this article, where the verification range is slightly extended to  the 80$^{th}$ maximal prime gap, approximately $1.836\times 10^{19}$.

\bibitem{Wells}
David Wells,
\emph{Prime numbers: the most mysterious figures in math},\\
(John Wiley, Hoboken 2005).

\bibitem{Sandor}
J\'oszsef S\'andor, ``On certain sequences and series with applications in prime number theory'',
Gaz. Mat. Met. Inf, {\bf6} (1985) 1--2.

\bibitem{Lowry-Duda}
David Lowry-Duda,
``A Short Note on Gaps between Powers of Consecutive Primes'',\\
arXiv:1709.07847 [math.NT]

\bibitem{Kristyan:2017}
Sandor Kristyan,
``On the statistical distribution of prime numbers: A view from where the distribution of prime numbers are not erratic'',
AIP Conference Proceedings {\bf1863}  (2017) 560013; \url{https://doi.org/10.1063/1.4992696} [arXiv:1709.02439  [math.NT]]


\bibitem{Kristyan:2018}
Sandor Kristyan,
``Note on the cardinality difference between primes and twin primes and its impact on function $x/\ln(x)$ in prime number theorem'',
AIP Conference Proceedings {\bf1978} (2018) 470064; \url{https://doi.org/10.1063/1.5044134}



\clearpage
%========================================================
% RECENT
%========================================================
\bibitem{Carneiro:2017}
Emanuel Carneiro, Micah B. Milinovich, Kannan Soundararajan,\\
``Fourier optimization and prime gaps'', to appear in Comment. Math. Helv., 
arXiv:1708.04122 [math.NT]
%========================================================
% HISTORY
%========================================================
\bibitem{Cramer:1919}
Harald Cram\'er, 
``Some theorems concerning prime numbers'',\\
Ark. Mat. Astron. Phys. {\bf15} (1920) 5.

\bibitem{Cramer:1936}
Harald Cram\'er, \\
``On the order of magnitude of the difference between consecutive prime numbers'', 
Acta Arithmetica  {\bf2} (1936) 23--46.

 %========================================================
\bibitem{Goldston}
Daniel Alan Goldston,
 ``On a result of Littlewood concerning prime numbers'', \\
 Acta Arithmetica XL {\bf3} (1982) 263--271.
%========================================================

%========================================================
% EXPLICIT PRIME GAPS
%========================================================

\bibitem{Dusart:2010}
Pierre Dusart,
``Estimates of some functions over primes without RH'',\\
	arXiv:1002.0442 [math.NT]

%\clearpage
\bibitem{Trudigan:2014}
Tim Trudgian,
``Updating the error term in the prime number theorem'',\\
Ramanujan J {\bf39} (2016) 225. \url{https://doi.org/10.1007/s11139-014-9656-6}
	[arXiv:1401.2689 [math.NT]].

\bibitem{Dusart2}
Pierre Dusart,
``Explicit estimates of some functions over primes'', \\
Ramanujan J {\bf45} (2018) 227. \url{https://doi.org/10.1007/s11139-016-9839-4}

\bibitem{Axler:2017}
Christian Axler,
``New estimates for some functions defined over primes'',
arXiv:1703.08032 [math.NT]

%========================================================
% MAXIMAL PRIME GAPS
%========================================================
 \bibitem{g80}
 Much of the discussion and discovery announcements regarding the numerical determination of maximal prime gaps now seems to be carried out online.
 
 For all of the maximal prime gaps up to $(80,g^*_{80},p^*_{80})$ see (as of 25 November 2018):
 
\url{https://en.wikipedia.org/wiki/Prime\_gap}
 
\url{http://trnicely.net/\#Maximal}
 
For all of the maximal prime gaps up to $(75,g^*_{75},p^*_{75})$ see (as of 25  November 2018):

\url{http://primerecords.dk/primegaps/maximal.htm} 
 
\url{https://primes.utm.edu/notes/GapsTable.html} 
 
(See also the sequences A005250, A002386, A005669, A000101, A107578.)

%========================================================
%========================================================
%========================================================
\bigskip
\hrule
%========================================================
\end{thebibliography}
\end{document}